\documentclass{article}

\usepackage{amsthm}
\usepackage{amsmath}
\usepackage{amssymb}
\usepackage{hyperref}
\usepackage{tikz-cd}
\usetikzlibrary{intersections}
\usetikzlibrary{hobby}
\usepackage{mathrsfs}
\usepackage{array}
\usepackage{import}
\usepackage{pdfpages}
\usetikzlibrary{fpu}
\usepackage[left=1.5in]{geometry}

\usepackage{titlesec}
\titlespacing*{\section}
{0pt}{5.5ex plus 1ex minus .2ex}{4.3ex plus .2ex}
\titlespacing*{\subsection}
{0pt}{5.5ex plus 1ex minus .2ex}{4.3ex plus .2ex}

\theoremstyle{definition}
\newtheorem{thm}{Theorem}
\newtheorem{prop}{Proposition}[section]
\newtheorem*{proposition*}{Proposition}
\newtheorem{defn}{Definition}[section]
\newtheorem{corollary}{Corollary}[section]

\newtheorem{remark}{Remark}[section]
\newtheorem{lemma}{Lemma}[section]

\newcommand{\R}{\mathbb{R}}
\newcommand{\C}{\mathbb{C}}

\def\x{\noexpand\x}    
\edef\weierstrass{0}     
\edef\currentbn{1}        
\foreach \i in {1,...,19} {
    \global\edef\currentbn{2*\currentbn}    
    \global\edef\weierstrass{\weierstrass + (1/(\currentbn)*cos((\currentbn*\x) r))}    
}

\newcommand{\Address}{{
  \bigskip
  \footnotesize 
  \noindent\textsc{Lehrstuhl für Analysis und Geometrie, Universität Augsburg, D-86135 Augsburg, Germany}\\
  \textit{E-mail address:}\href{mailto:filipvbrocic@gmail.com}{\tt filipvbrocic@gmail.com}
  \vspace{5mm}
 
 \noindent\textsc{School of Mathematical Sciences, Tel Aviv University, Ramat Aviv, Tel Aviv 69978
Israel}\\
 \textit{E-mail address:}\href{mailto:smatijevic@tauex.tau.ac.il}{\tt smatijevic@tauex.tau.ac.il}
}}

\title{A note on the wild symplectic ellipsoids}
\author{Filip Bro\'ci\'c and Stefan Matijevi\'c}
\begin{document}

\maketitle

\begin{abstract}
We show that the symplectic $2$-product of $n$ two-dimensional star-shaped domains has an interior symplectomorphic to that of a symplectic ellipsoid. Adapting this construction, given $0<\alpha \leq 1$, we obtain that every open subset of $\R^{2n}$ with a smooth boundary is symplectomorphic to an open set whose boundary contains a set of Hausdorff dimension $2n-1+\alpha$.
 
\end{abstract}

\section{Introduction}

We say that $W \subset (\mathbb{R}^{2n}, \omega_0 = \sum_i dx_i \wedge dy_i)$ is a \emph{star-shaped domain} if it is the closure of an open, bounded set that is star-shaped with respect to the origin. Given two star-shaped domains $K \subset \R^{2k}$ and $T \subset \R^{2\ell}$, the symplectic $p$-product of $K$ and $T$ is defined by:
$$
K \times_p T := \bigcup_{t \in [0,1]} t^{\frac{1}{p}} K \times (1-t)^{\frac{1}{p}} T \subset \R^{2(k+\ell)},
$$
where, for $A \subset \R^n$ and $s>0$, the set $sA$ is given by $\{sx \mid x \in A\}$. The symplectic $p$-product was introduced and studied in \cite{KO23}. This operation is associative, that is, for $K \subset \R^{2k}$, $T \subset \R^{2\ell}$, and $G \subset \R^{2m}$, we have:
\[
K \times_p T \times_p G := (K \times_p T) \times_p G = K \times_p (T \times_p G)
= \bigcup_{\substack{t_1,t_2,t_3 \in [0,1]\\ t_1 + t_2 + t_3 = 1}} t_1^{\frac{1}{p}} K \times t_2^{\frac{1}{p}} T \times t_3^{\frac{1}{p}} G.
\]
Hence, the $p$-product of $n \in \mathbb{N}$ star-shaped domains is well-defined. In the Lagrangian setting, the $p$-product was explored in \cite{OR22, Br25}. In this paper, we will be specifically interested in the symplectic $2$-product.  

For $a_1, \dots, a_n > 0$, the symplectic ellipsoid $E(a_1, \dots, a_n) \subset \R^{2n} \cong \C^n$ is defined by
$$
E(a_1, \dots, a_n) := \left\{ (z_1, \dots, z_n) \;\middle|\; \sum_i \frac{\pi |z_i|^2}{a_i} \leq 1 \right\}.
$$

\begin{thm}\label{thm:main}
Let $W_1, W_2, \dots, W_n \subset \R^2$ be star-shaped domains with respective areas $a_1, \dots, a_n$. Then the interior of $W_1 \times_2 W_2 \times_2 \cdots \times_2 W_n$ is symplectomorphic to the interior of the symplectic ellipsoid $E(a_1, a_2, \dots, a_n)$.
\end{thm}

The boundary of a generic star-shaped domain is highly irregular. In fact, the set of star-shaped domains whose boundary is nowhere differentiable is generic in the Baire category sense. Hence, this theorem provides examples of domains that do not have a well-defined tangent space at any point by taking the $2$-product of generic star-shaped domains $W_i \subset \R^2$, yet these domains have interiors symplectomorphic to the interiors of symplectic ellipsoids. In particular, this implies that the boundary does not admit a characteristic foliation (even in the generalized sense \cite{Cl81, Cl83}), despite the interior being symplectomorphic to that of a symplectic ellipsoid. Furthermore, there exist curves bounding star-shaped domains in $\R^2$ whose Hausdorff dimension is $2$. An explicit construction of a function which leads to such an example can be found in \cite{XZ07}. In \cite{FH11}, it was shown that the set of such curves is large in an appropriate sense, even though it is not generic.

\begin{corollary}\label{cor:haus}
Given positive real numbers $a_1,\dots, a_n$ and $\alpha \in (0,1]$, there exists a compact star-shaped domain $W \subset \R^{2n}$ whose boundary has Hausdorff dimension $2n-1+\alpha$ and whose interior is symplectomorphic to the interior of $E(a_1, \dots, a_n)$.
\end{corollary}

In \cite{CGH25}, examples are given of compact domains in $\C^2$ whose boundaries have Minkowski dimension arbitrarily close to $4$ and whose interiors are symplectomorphic to the interior of an unbounded toric domain. For a certain class of unbounded toric domains, the boundary has Minkowski dimension bounded from below by a constant strictly bigger than $3$. In general, the Hausdorff dimension is a lower bound for the Minkowski dimension (see \cite[Proposition 1.1]{FH11}). For $0<\alpha<1$, we use the graphs of a modified Weierstrass functions from \cite{Hu98}, which have Hausdorff dimension $1+\alpha$. By modifying the functions in the proof slightly, for any $0<\alpha\leq1$, Corollary \ref{cor:haus} also provides examples of compact domains $W \subset \C^n$ whose boundaries have Minkowski dimension $2n-1+\alpha$, and whose interiors are symplectomorphic to the interior of a symplectic ellipsoid. Symplectic ellipsoids are classic examples of bounded toric domains that are both convex and concave in the sense of \cite{CG19}. Note that our theorem illustrates a different phenomenon than \cite{CGH25}: we start from a domain that has a smooth boundary, and ask how irregular can be a boundary of a domain that has a symplectomorphic interior to the starting domain, whereas in \cite{CGH25} they show that for certain domains, the boundary can not be made more regular. In the spirit of making the boundary more irregular, after localizing the construction in Theorem \ref{thm:main}, we obtain:

\begin{thm}\label{thm:domain}
Let $U \subset \R^{2n}$ be an open set whose topological boundary has a point $p$ with a
smooth neighborhood. For every $\alpha \in (0,1]$, there exists a symplectomorphism $\Psi: U \to W$ where $W\subset \R^{2n}$ is an open set whose boundary contains an open set of Hausdorff dimension $2n-1+ \alpha$.
\end{thm}

\begin{remark}
By inspecting the proof of Lemma \ref{lma:smooth}, one can easily see that Theorem \ref{thm:main} also holds in the $p$-product setting, where, instead of $E(a_1,\dots,a_n)$, we consider the interior of $\mathbb{D}(a_1) \times_p \cdots \times_p \mathbb{D}(a_n)$, where $\mathbb{D}(a_i)$ is a disc centered at the origin of area $a_i$. In fact, one can define the $p$-product of two Liouville domains $(W_1, \lambda_1)$ and $(W_2, \lambda_2)$ as the subset:
\[
\{r_1^{p/2} + r_2^{p/2} \leq 1\} \subset \widehat{W}_1 \times \widehat{W}_2,
\] 
where $\widehat{W}_i$ are the completions $W_i \sqcup_{\partial W_i} \partial W_i \times [1, +\infty)$ of $W_i$. Here, $r_i$ are $1$-homogeneous functions on $\widehat{W}_i$, continuously extended by $0$ on the $\mathrm{Core}(W_i, \lambda_i)$, and such that $W_i = \{ r_i \leq 1\}$. Now, if $K_i \subset \widehat{W}_i$ is the image of a Hamiltonian flow invariant under the Liouville flow, then the interiors of $W_1 \times_p W_2$ and $K_1 \times_p K_2$ are symplectomorphic by arguments analogous to those in Section \ref{sec:main_proof}.
\end{remark}

As a contrast to Theorem \ref{thm:main}, it was pointed out to us by Alberto Abbondandolo that the interior of the rational polydisc $\mathbb{D}(a) \times \mathbb{D}(b)$ is never symplectomorphic to the interior of a smooth star-shaped domain. The argument is originally due to Oliver Edtmair (\cite{Ed25}) and relies on $ECH$ capacities. Namely, $\mathbb{D}(a) \times \mathbb{D}(b)$ has the same $ECH$ capacities as $E(a,b)$. But if a smooth star-shaped domain $W$ has the same $ECH$ capacities as the rational ellipsoid $E(a,b)$, then $W$ is symplectomorphic to $E(a,b)$. However, $\mathbb{D}(a) \times \mathbb{D}(b)$ is not symplectomorphic to $E(a,b)$ as can be seen with Ekeland-Hofer capacities.

\subsection{Dynamics on sufficiently regular boundaries and the Zoll property in the convex setting}

In this section, we discuss the dynamics of the generalized characteristic foliation on domains with sufficiently regular non-smooth boundaries. For instance, using the construction from \cite{AK70a, AK70b}, one obtains strongly convex domains whose interiors are symplectomorphic to those of irrational ellipsoids, yet whose boundary dynamics are quite different (see \cite[Appendix]{ABE25}). In our setting, when the dynamics is well-defined, we have the following:

\begin{prop}\label{prop:dyn}
Let $W_i \subset \R^2$ be compact star-shaped domains of area $a_i$, with Lipschitz boundaries and such that each ray from the origin intersects the boundary at a single point. Then the dynamics of the characteristic foliation is well-posed and is topologically conjugate to the Reeb flow on $E(a_1, \dots, a_n)$.
\end{prop}

In \cite{GGM21}, it was shown that a smooth strongly convex domain is Zoll if and only if the first Ekeland-Hofer spectral invariant (see \cite{EH87}) coincides with the $n$-th one (see also \cite{GRT25}). In \cite{Ma24}, it was further shown that the Gutt-Hutchings capacities coincide with these invariants. Consequently, being Zoll in a smooth, strongly convex setting is a symplectic notion; that is, it depends only on the interior of the convex body. Moreover, in the convex setting, the notion of the Zoll property can be extended topologically, as introduced in \cite{Ma25}. Specifically, one can show that the Fadell-Rabinowitz index of the space of generalized systols of $\partial K \subset \C^n$ being at least $n$ is equivalent to $c_1^{\mathrm{GH}}(K) = c_n^{\mathrm{GH}}(K)$ (see \cite{GH18}). The generalized Zoll property was further investigated in \cite{HK25}. When the characteristic dynamics is well-posed, this condition on the Fadell-Rabinowitz index is equivalent to the boundary being foliated by systoles (see \cite[Theorem 1.7]{Ma25}).  In particular, the $2$-product $W_1 \times_2 \cdots \times_2 W_n$ of $n$ convex bodies of equal area is generalized Zoll. Moreover, from Proposition \ref{prop:dyn}, the characteristic dynamics on the boundary of $W_1 \times_2 \cdots \times_2 W_n$ is foliated by systols when $W_i$ are star-shaped with Lipschitz boundaries and of equal area.

A convex body $K \subset \C^{n}$ \emph{cuts additively} if, for every hyperplane $H$ that separates $K$ into $K_1$ and $K_2$, we have
$$
c_{\mathrm{EHZ}}(K) = c_{\mathrm{EHZ}}(K_1) + c_{\mathrm{EHZ}}(K_2).
$$
If the boundary of a generalized Zoll convex body $K$ has well-defined characteristic dynamics, it was shown in \cite[Theorem A]{HK25} that $K$ cuts additively. Combining Proposition \ref{prop:dyn} and \cite[Theorem A]{HK25} leads to:

\begin{corollary}
If $W_i$ are convex bodies whose interiors contain the origin with the same area, then $W_1 \times_2 \cdots \times_2 W_n$ cuts additively.  
\end{corollary}

Note that there are examples of convex curves that are not differentiable on a (countable) dense subset. 

\subsection{On the Zoll property beyond convexity}

Corollary \ref{cor:haus} in particular implies that it is not possible to extend the Zoll property dynamically in such a way that includes all star-shaped domains whose interior is symplectomorphic to the interior of the ball. It is known by \cite{ABE25} that Zoll domains are maximizers of the systolic ratio in $C^2$ topology but not in $C^0$-topology in a star-shaped setting. Hence, the condition on the systolic ratio being locally maximized can't serve us for the classification of highly irregular star-shaped domains with an interior symplectomorphic to the open ball. 

Note that $c_1^{\mathrm{GH}} (W) = c_n^{\mathrm{GH}}(W)$ still holds by its symplectic nature for $W = W_1 \times_2 \cdots \times_2 W_n$ whenever $W_i \subset \R^2$ are star-shaped domains of the same area. The condition $c_1^{GH}(W)=c_n^{GH}(W)$ is indeed a good candidate to help the classification of star-shaped domains whose interior is symplectomorphic to the interior of a ball (see \cite{GRT25}), but it is not clear how to interpret it if the boundary is of too low regularity. 

Let $W\subset \mathbb{R}^{2n}$ be a star-shaped domain. Following \cite{Ca25}, we say that $W$ is \textit{boundary minimal}\footnote{Note that in \cite{Ca25} the boundary minimal condition was defined with respect to the spectral capacity instead of $c_1^{EH}$.} for the first Ekeland-Hofer capacity $c_1^{EH}$ if for every open $U$ such that $\partial W \cap U \neq \emptyset$ it holds $c^{EH}_{1}(W \setminus U)<c_1^{EH}(W)$.  It is convenient to use $c_1^{EH}$ for the definition, since it is well defined for all subsets of $\R^{2n}$. For the definition of $c_1^{EH}$ see \cite{EH89, EH90}. In \cite{GR24} it was shown that for star-shaped domains $W$ with smooth boundary, Gutt-Hutchings capacities coincide with Ekeland-Hofer capacities.

In \cite[Theorem 5]{Ca25} it was shown that for a strongly convex domain $K$, being Zoll is equivalent to $K$ being boundary minimal. In that regard, it would be interesting to understand whether the symplectic $2$-product of domains with the same areas is boundary minimal. From Theorem \ref{thm:main} it follows:

\begin{prop}\label{prop:boundary_minimal}
If $W_i$ are star-shaped domains with the same area, such that a radial ray intersects the boundary in a single point, then $W_1 \times_2 \cdots \times_2 W_n$ is boundary minimal.
\end{prop}

Note that in the previous proposition, the boundaries of $W_i$ can be highly irregular, and yet one can understand its boundary minimal property. This suggests a possible generalization of the Zoll property to arbitrary star-shaped domains. It seems beneficial to further investigate the boundary minimal property.

\subsection*{Acknowledgements}
We would like to thank Dan Cristofaro-Gardiner, Richard Hind, Alberto Abbondandolo, Oliver Edtmair, and Dylan Cant for useful comments on the earlier version of this paper. F.B. is supported by the Deutsche Forschungsgemeinschaft (DFG, German Research Foundation) – 517480394. S.M. is partially supported by the ISF-NSFC grant 3231/23 and partially supported by the ISF grant 938/22.
 
\section{Proofs}

\subsection{Proof of Theorem \ref{thm:main}}\label{sec:main_proof}
\begin{defn} A compact domain $W \subset \R^{2n}$ star-shaped with respect to origin is \textit{nice} if it has a smooth boundary $\partial W$, and the radial vector field $X = 1/2\sum x_i \partial_{x_i} + y_i \partial_{y_i}$ is transverse to $\partial W$.

\end{defn}

\begin{lemma}\label{lma:ham}
Let $W \subset \R^2$ be a nice star-shaped domain with area $a$. There exists a $1$-homogeneous Hamiltonian isotopy $\Phi^t_{H_t}$ on $\R^{2} \setminus \{0\}$ such that $\Phi^1_{H_t} (\mathbb{D}(a) \setminus \{0\}) = W \setminus \{0\}$
\end{lemma}

\begin{proof}

 Let $S^1(a ) := \partial \mathbb{D}(a )$ be a boundary of the disk of area $a $, and let $\alpha = \lambda_{0} \vert_{S^1(a )} $, where $\lambda_0 = 1/2(x dy - y dx)$. 
Consider the symplectization $S^1(a ) \times (0, +\infty)$ of $(S^1(a ), \alpha )$ with a symplectic form $d(r \alpha)$. There is a symplectomorphism $S:\R^2 \setminus \{0\} \to S^1(a ) \times (0, +\infty)$ given by: $$r e^{i \theta} \mapsto (\sqrt{a  / \pi} e^{i \theta}, \pi/a  r^2).$$ Symplectomorphism $S$ maps $\partial W $ to $\{(x, h_{W }(x)) \mid x \in S^1(a )\}$ for the unique positive smooth function $h_{W }:S^1(a ) \to \R$, i.e., $S$ induces a strict contactomorphism from $(\partial W , \lambda_0\vert_{W })$ to $(S^1(a ), h_{W } \alpha  )$. 

Since $W $ and $\mathbb{D}(a )$ have the same area, $(S^1(a ), \alpha )$ has the same length as  $(S^1(a ), h_{W} \alpha)$. By Moser's trick, there is a volume preserving isotopy $\varphi^t$ on $S^1(a )$ such that $(\varphi^1)^* (h_W  \alpha ) = \alpha $. This isotopy is generated by the vector field $X_t $ on $S^1(a )$ with the contact Hamiltonian $f_t : S^1(a ) \to \R$, defined by $f_t (x) = \alpha (X_t)$. Now, each contact Hamiltonian $f_t$ gives rise to a Liouville equivariant Hamiltonian $H_t (x,r) = f_t (x) r$ on the symplectization whose flow is $1$-homogeneos lift $\Phi^t_{H_t}$ of the contact isotopy generated by $X_t$. After identifying the symplectization with $\R^2 \setminus \{0\}$, $\Phi^t_{H_t}$ can be extended continously on $\R^2$ to $0$ by setting $\Phi^t_{H_t}(0)=0$. Homeomorphism $\Phi^1_{H_t}$ maps $\mathbb{D}(a )$ to $W $.
\end{proof}

\begin{lemma}\label{lma:smooth} Let $W_1, \dots, W_n$ be nice star-shaped domains with respective areas $a_1, \dots, a_n$. For every $ \epsilon > 0$ small enough there exists a symplectomorphism $\Psi(z_1,\dots,z_n) = (\psi_1(z_1), \dots, \psi_n(z_n))$ such that  $\psi_i(0) = 0$ for every $i\in \{1,\dots,n\}$ and: 
$$(1-\epsilon) (W_1 \times_2 W_2 \times_2 \cdots \times_2 W_n ) \subset \Psi (E(a_1, \dots, a_n))\subset (1+\epsilon) (W_1 \times_2 W_2 \times_2 \cdots \times_2 W_n).$$
\end{lemma}

\begin{proof}

For each $W_i$ let $H_t^i$ be the Hamiltonians that generate the $1$-homogeneous flows from Lemma \ref{lma:ham}. Fix $0 <\epsilon <1$ arbirtrarily small, and let $\delta_i>0$ be such that: $\Phi^t_{H^i}(\mathbb{D}(\delta_i) ) \subset \epsilon' W_i,$ where $\epsilon' < \sqrt{\epsilon/n}$. Note that Hamiltonian isotopy $\Phi^t_{H^i}$ of $\R^2 \setminus \{0\}$ was extended continuously to $\R^2$ by $\Phi^t_{H^i}(0) = 0$. Define the time dependent Hamiltonian $H_t: \R^{2n} \to \R$ by: 
$$
H_t(z_1, \dots., z_n) = \sum_{i=1}^n \rho_i(|z_i|^2) H_t^i(z_i),
$$
where $\rho_i :\R \to [0,1]$ is a smooth cut-off function such that $\rho_i \equiv 1$ for $t \geq \delta_i / \pi$, and $\rho_i \equiv 0$ near $0$. Let $\Phi^t_{H}$ be the Hamiltonian flow of $H_t$, and set $\Psi:= \Phi^1_{H}$.
\begin{itemize}
\item $\Psi(E(a_1, \dots, a_n)) \subset (1+\epsilon) W_1 \times_2 \cdots \times_2 W_n$.

 W.l.o.g we can assume that $\pi |z_i|^2 < \delta_i$ for $1 \leq i \leq k$, and $\pi |z_i|^2 \geq \delta_i$ for $i \geq k+1$ for $0 \leq k \leq n$. For $ z \in E(a_1, \dots, a_n)$ there are $\lambda_i \in [0,1]$, and there are $\zeta_i \in \mathbb{D}(a_i)$ such that $z = (\lambda_1 \zeta_1, \dots, \lambda_n\zeta_n)$ and $\lambda_1^2 + \cdots +\lambda_n^2 = 1$. We have: $$\Psi(z) = (\psi_1(\lambda_1 \zeta_1), \dots, \psi_n(\lambda_n \zeta_n)) \in \epsilon' W_1 \times \epsilon' W_2 \times \cdots \times \epsilon' W_k \times  \lambda_{k+1} W_{k+1} \times \cdots \times \lambda_{n}  W_{n},
$$
where we have used that $\psi_i(\mathbb{D}(\delta_i)) \subset \epsilon' W_i$ and $\psi_i$ is $1$ - homogeneous on the set $\pi |z_i|^2 \geq \delta_i$.

From inequality: $$k \epsilon'^2 + \lambda_{k+1}^2 + \cdots \lambda_{n}^2 \leq \epsilon+ 1,$$ we get: $$\epsilon' W_1 \times \cdots \times \epsilon' W_k \times  \lambda_{k+1} W_{k+1} \times \cdots \times \lambda_{n}  W_{n}  \subset (1+\epsilon) (W_1 \times_2 \cdots \times_2 W_n).$$

\item $(1 - \epsilon) W_1 \times_2 \cdots \times_2 W_n \subset \Psi(E(a_1, \dots, a_n))$.

For $w \in (1-\epsilon) (W_1 \times_2 \cdots \times_2 W_n)$, there are $\lambda_i \in [0,1]$ and $w_i \in W_i$ such that $w = (1-\epsilon) (\lambda_1 w_1, \dots, \lambda_n w_n)$ and $\lambda_1^2 + \cdots +\lambda_n^2 =1$. Since $\psi_i(\mathbb{D}(a_i)) = W_i$, let $z_i$ be such that $\psi_i(z_i) = (1-\epsilon) \lambda_i w_i$. When $\pi |z_i|^2 \geq \delta_i $ we have $\frac{\pi |z_i|^2}{a_i} \leq (1- \epsilon)^2 \lambda_i^2$ because we are in the region where $\psi_i$ are $1$-homogeneous, and $\psi_i$ preserves the area. More precisely, $\psi_i(\mathbb{D}(\pi |z_i|^2)) \subseteq (1-\epsilon) \lambda_i W_i$, and area of $(1-\epsilon) \lambda_i W_i$ is $(1-\epsilon)^2\lambda_i^2 a_i$. In other case, we have $\pi |z_i|^2 < \delta_i$, but $\delta_i < a_i \epsilon'^2$ since $\psi_i(\mathbb{D}(\delta_i)) \subset \epsilon' W_i$. We conclude that:

$$
\frac{\pi |z_1|^2 }{a_1} + \cdots + \frac{\pi |z_n|^2 }{a_n} \leq n \epsilon'^2 + (1-\epsilon)^2 (\lambda_1^2 + \cdots + \lambda_n^2) < \epsilon + (1-\epsilon)^2 <1,
$$
since $0 < \epsilon < 1$.

In particular, since $\psi_i(0)=0$ for all $i$, we have that $(1-\epsilon) W_i \subset \psi_i(\mathbb{D}(a_i)) \subset (1+\epsilon)W_i$ by setting $z_j=0$ for $j \neq i$.
\end{itemize}

\end{proof}

\begin{proof}[Proof of Theorem \ref{thm:main}]
Fix an increasing sequence $\delta_k$ converging to $1$, and a sequence $\delta_k'$ such that $\delta_k < \delta_k' < \delta_{k+1}$. For each $W_i$, consider an exhausting nested sequence $W_i^k \subset W_i$ of nice star-shaped domains in $\R^2$ with areas $\delta_k a_i$ (see Figure \ref{fig:nested}). 

\begin{figure}[h!]
\centering
\begin{tikzpicture}
\draw[black!30!white, use Hobby shortcut,closed=true]
(0, 1) .. (0.35, 0.35) .. (1, 0) .. (0.35, -0.35) .. (0, -1) .. (-0.35, -0.35) .. (-1, 0) .. (-0.35, 0.35) .. (0,1);
\draw[black!40!white, use Hobby shortcut,closed=true]
(0, 1.15) .. (0.45, 0.45) .. (1.15, 0) .. (0.45, -0.45) .. (0, -1.15) .. (-0.45, -0.45) .. (-1.15, 0) .. (-0.45, 0.45) .. (0,1.15);

\draw[black!50!white, use Hobby shortcut,closed=true]
(0, 1.3) .. (0.55, 0.55) .. (1.3, 0) .. (0.55, -0.55) .. (0, -1.3) .. (-0.55, -0.55) .. (-1.3, 0) .. (-0.55, 0.55) .. (0,1.3);

\draw[black!60!white, use Hobby shortcut,closed=true]
(0, 1.35) .. (0.65, 0.65) .. (1.35, 0) .. (0.65, -0.65) .. (0, -1.35) .. (-0.65, -0.65) .. (-1.35, 0) .. (-0.65, 0.65) .. (0,1.35);

\draw (-0.3, 1.37) -- (0.3, 1.37) -- (0.7, 0.7)-- (1.37, 0.3) -- (1.37, -0.3) -- (0.7, -0.7) -- (0.3, -1.37) -- (-0.3, -1.37) -- (-0.7, -0.7) -- (-1.37, -0.3) -- (-1.37, 0.3) -- (-0.7, 0.7) -- (-0.3, 1.37);

\draw[black!60!white] node at (0,0) {$W_i^k$};
\draw node at (0.7, 0.7) [above right] {$W_i$};

\draw (-4,0) circle (1.25); 
\draw[black!60!white] (-4,0) circle (1.15);
\draw[black!50!white] (-4,0) circle (1.05);
\draw[black!40!white] (-4,0) circle (0.95);
\draw[black!30!white] (-4,0) circle (0.85);
\draw[black!60!white] node at (-4,0) {$\mathbb{D}(\delta_k a_i)$};
\draw node at (-3.2,0.7) [above right] {$\mathbb{D}(a_i)$};
\draw[->] (-2.3,0) -- (-1.8, 0);
\draw node at (-2.05, 0) [above] {$\psi$};
\end{tikzpicture}
\caption{Exhausting nested sequence $W_i^k$.}
\label{fig:nested}
\end{figure}
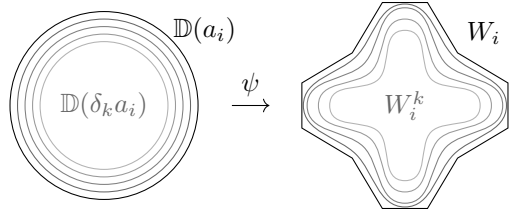

After approaprietely choosing $\epsilon_k$, we can apply Lemma \ref{lma:smooth} to obtain a sequence $\Psi_k = (\psi_1^k, \dots, \psi_n^k)$ of symplectomorphisms such that
$$
W_1^k \times_2 \cdots\times_2 W_n^k \subset \Psi_k(\delta'_k E(a_1, \dots, a_n)) \subset W_1^{k+1}\times_2 \cdots \times_2 W_2^{k+1}.
$$
Now, the space of symplectic embeddings of a disc into a compact connected surface is connected (see, e.g., \cite[Proposition A.A]{AS01}), and every path can be realized by the ambient isotopy. Using $\mathrm{Im} \psi_i^k \subset \mathrm{Im} \psi_i^{k+1}$, we can modify each component $\psi_i^{k+1}$ such that $\psi_i^{k+1} \vert_{\mathbb{D}(\delta_k' a_i)} = \psi_i^k$, and $\psi_i^{k+1}$ remains unchanged near the boundary of $\mathbb{D}(\delta_{k+1}' a_i)$.  We define a symplectomorphism 
$$
\begin{aligned}
\Psi: \mathrm{int}(E(a_1, \dots, a_n)) &\to \mathrm{int}(W_1 \times_2 \cdots \times_2 W_n) \\
z &\mapsto \Psi_k(z),
\end{aligned}
$$
where $k$ is such that $z \in \delta_k E(a_1, \dots, a_n)$. For a similar application of the exhaustion trick, see \cite[Lemma 4.3]{LMS13}.

\end{proof}

\subsection{Proof of Corollary \ref{cor:haus}}
For the definition of Hausdorff dimension, see \cite[Section 3]{Sch07}. It follows from \cite[Theorem 2]{Sch07} and \cite{He71} that Hausdorff dimension has the following properties:
\begin{itemize}
\item If $X \subset Y$ then $\dim_{\mathrm{H}} (X) \leq \dim_{\mathrm{H}} (Y)$.
\item $\dim_{\mathrm{H}} (X \times Y) \geq \dim_{\mathrm{H}} (X) + \dim_{\mathrm{H}} (Y)$.
\item For $X \subset \R^{d}$ we have $\dim_{\mathrm{H}} (X) \leq d$.
\item $\dim_{\mathrm{H}}(X \times \R^d) = \dim_{\mathrm{H}}(X)+d$.
\end{itemize}

Let $W_1$ be a compact star-shaped domain with boundary of Hausdorff dimension $1+\alpha$, and of area $a_1$. If $\alpha<1$ we can construct $W_1$ by taking a Weirestrass function $W_{a,b}(x) = \sum_{n=0}^{+\infty} a^n \cos (2 \pi b^n x)$ for suitable $0<a<1<b$, since the graph of $W_{a,b}$ has Minkowski\footnote{It is conjectured that the Weirstrass function $W_{a,b}$ also has Hausdorff dimension $2+\log a/ \log b$. This is confirmed for certain parameters in \cite{BBR14}.} dimension $2+\log a/ \log b$ (see \cite{KMPY84}). For a domain with Hausdorff dimension $1+\alpha = 2+\log a/ \log b$, we can take a small phase shift of the Weierstrass function as in \cite{Hu98}:

$$
W_{a,b, \theta}(x) = \sum_{n=0}^{+\infty} a^n \cos (2 \pi(b^n x + \theta_n)),
$$ 

for a suitable choice of $\theta_n$. For $\alpha=1$, we can use the function from \cite{XZ07}. The construction is similar to Wierstrass function: $$f(x) = \sum_{n=1}^{\infty} a^{n^{\alpha}} \phi(a^{-n^{\beta}} x),$$
where $0<a<1$, $1<\alpha <\beta$ and $\phi(x)=2x$, for $0 \leq x \leq 1/2$, and it is extended to $\R$ by $\phi(-x)=\phi(x)$, and $\phi(x+1)=\phi(x)$. Let $H_{W_1}$ be a $2$-homogeneous function such that $W_1 = \{H_{W_1} \leq 1\}$. It is given by $H_{W_1} (z) = f(z/|z|) |z|^2$ for a unique function $f:S^1 \to (0, +\infty)$, whose graph has dimension $1+\alpha$. 

Consider $W:=W_1 \times_2 E(a_2,\dots, a_n)$. From Theorem \ref{thm:main}, the interior of $W$ is symplectomorphic to the interior of $E(a_1,\dots,a_n)$. The defyning $2$-homogeneous function for $W$ is:
$$
H_W(z_1,\dots, z_n) = H_{W_1}(z_1) + \frac{\pi |z_2|^2}{a_2} \cdots + \frac{\pi |z_n|^2}{a_n},
$$
and, the boundary of $W$ is given by $H_W^{-1}(1)$. Given any point  on $\vec{z}=(z_1,\dots,z_n) \in \partial W$, we have two cases: $z_n=0$ and $z_n \neq 0$. Since $z_n=0$ cuts the dimension by two, it is enough to consider the case $z_n \neq 0$. We can see a neighborhood of $\vec{z} \in \partial W$ in polar coordinates as a graph of the function:
$$
\begin{aligned}
r_n : U &\to \R,\\
(r_1, \theta_1, r_2, \theta_2, \dots, \theta_{n-1}, \theta_n) &\mapsto \frac{a_n}{\pi} \sqrt{1- f(e^{i\theta_1})r_1^2 - \cdots - \frac{\pi r_{n-1}^2}{a_{n-1} }},
\end{aligned}
$$

where $U \subset \R^{2n-1}$. Since the square root is smooth outside of the origin, the dimension of the graph is determined by:

$$
1- f(e^{i\theta_1})r_1^2 - \cdots  - \frac{\pi r_{n-1}^2}{a_{n-1} }.
$$
Now, for a smooth function $g$ and any $h$, we have $\dim_{\mathrm{H}}(\Gamma_h) = \dim_{\mathrm{H}}(\Gamma_{h+g})$, where $\Gamma_{h}$ is the graph of a function. The reason is that the map $(x,y) \mapsto (x, y + g(x))$ is a diffeomorphism, and hence, it preserves the Hausdorff dimension. Similar arguments apply for the product with a positive smooth function, i.e., for a positive smooth function $g$ and any $f$, we have $\dim_{\mathrm{H}}(\Gamma_h) = \dim_{\mathrm{H}}(\Gamma_{hg})$. Lastly, since we have reduced the question to the Hausdorff dimension of the graph of $f$ seen as a function of $2n-1$ variables, we appeal to the propery $\dim_{\mathrm{H}}(X \times \R^d) = \dim_{\mathrm{H}}(X)+d$ from \cite{He71}.  Applying these two observations, we conclude that the Hausdorff dimension of $\partial W$ is $2n-1+\alpha$. Note that he proof is valid in the region $r_1 > 0$, but $r_1=0$ cuts the dimension by $2$, hence the conclusion remains.


\subsection{Proof of Theorem \ref{thm:domain}}

For the proof, we will need the following lemma. Let $U$ be an open set, compactly contained in an open set $W$, where $W \subset (M, \omega)$ has a compact closure. Let $H: M \times [0,1] \to \R$, and $G: M \times [0,1]$ be two Hamiltonians, and let $\Phi^t_G$ and $\Phi^t_H$ be the flows that are generated by $G_t$ and $H_t$. 

\begin{lemma}\label{lma:maksim}
If $\Phi^t_H(\Bar{U}) \subset \Phi^t_G (W)$ for all $t \in [0,1]$ then there exists $\widetilde{G}_t: M \to \R$ such that:
\begin{itemize}
\item $\Phi^t_{\widetilde{G}}(x) = \Phi^t_{H}(x)$, for all $x \in U$,
\item $\Phi^t_{\widetilde{G}}(x) = \Phi^t_G(x)$, for all $x \in W^{\mathrm{c}}$.
\end{itemize}
\end{lemma}

\begin{proof}
Represent $\Phi^t_{H}$ as $\Phi^t_{H} = \Phi^t_G \circ ((\Phi^t_G)^{-1} \circ \Phi^t_H)$. Since $\Phi^t_H(\Bar{U}) \subset \Phi^t_G (W)$ we have that $\bigcup_t  (\Phi^t_G)^{-1} \circ \Phi^t_H (\Bar{U}) \subset W$, i.e., there are open sets $\Bar{U}' \subset W'$ such that: $$U \subset \bigcup_t  (\Phi^t_G)^{-1} \circ \Phi^t_H (\Bar{U}) \subset U'.$$ Set: $$\widetilde{G}_t = G_t \# ( \rho \Bar{G_t} \# H_t),$$
where $\rho$ is a cut-off function which is equal to $1$ on $\bigcup_t  (\Phi^t_G)^{-1} \circ \Phi^t_H (\Bar{U})$, and $\rho \equiv 0$ on $W'^{\mathrm{c}}$. Here $F\#K $ is a Hamiltonian that generates $\Phi^t_F \circ \Phi^t_K$, and $\Bar{F}$ is a Hamiltonian that generates $(\Phi^t_F)^{-1}$. The flow $\Phi^t_{\widetilde{G}}$ satisfies requirements of the lemma, since $\widetilde{G}_t$ coincides with $G_t$ outside of $W$ and for all $x \in U$ we have $\Phi^t_{\widetilde{G}}(x) = \Phi^t_H(x)$ by construction.
\end{proof}

The proof of Theorem \ref{thm:domain} is divided into two parts. The first part brings a smooth part of the boundary on the boundary of the ball, and the second part inserts a symplectomorphism $\psi: \mathrm{int} (B^{2n}(\epsilon) )\to \mathrm{int} (W)$ which does not change the ball outside of the neighborhood $V$ of $p \in \partial U$, and the $\partial (\psi(V) \cap W)$ has Hausdorff dimension $2n$.

\textbf{Part 1:} Let $\vec{n}$ be the unit outer normal to $\partial U$ at $p$. Consider a point $q \in U$ such that $\vec{n}=\sqrt{\delta / \pi}(\vec{p} - \vec{q})$ and half open segment satisfies $[q,p) \subset U$. Without loss of generality, we can assume that $q= 0$. Now, take the ball $B^{2n}( \delta)$ whose boundary is tangent $\partial U$ at $p$. From our choices, we conclude that there is $\epsilon < \delta$ so that $B^{2n}(\epsilon) \subset U$. Using the standard Liouvile form $\lambda_0$ on $B^{2n}(2\delta)$, there is a neighborhood $\mathcal{V} \subset \partial U$ of $p$ so that Liouville vector field\footnote{For a Liouville domain $(W, d\lambda)$, the Liouville vector field $X$ is a unique vector field such that $i_X d\lambda = \lambda $.} $X(x,y)=1/2 \sum_i x_i \partial_{x_i} + y_i \partial_{y_i}$ is transverse to $\mathcal{V}$ since: $$X(p) =\frac{1}{2}\vec{p}  =  \frac{1}{2} \sqrt{\frac{\pi}{\delta}}  \vec{n}.$$ By the contact Darboux theorem, we know that there is a neighborhood $\mathcal{V}' \subset \mathcal{V}$, a neighborhood $\mathcal{W}' \subset \partial B^{2n}( \epsilon)$ of $ \sqrt{\epsilon / \pi} p$ and a strict contactomorphism $\phi: (\mathcal{W}' , \lambda_0)\to (\mathcal{V}', \lambda_0)$, and $\phi( \sqrt{\epsilon / \pi} p) = p$. 

From the proof of the Darboux theorem, this contactomorphism can be realized as the isotopy $\varphi^t : \mathcal{W}' \to \mathcal{W'}$ that fixes $\sqrt{\epsilon / \pi} p$, such that on a smaller set $\mathcal{W}''$ we have $(\varphi^1)^* e^f \alpha_0 = \alpha_0$, where $e^f \alpha_0 =\phi^{*} \lambda_0$ and $\alpha_0 = \lambda_0 \vert_{\mathcal{W'}}$ for some $f: \mathcal{W'} \to \R$. Let $h_t$ be the contact Hamiltonian which generates $\varphi^t$. By cutting off $h_t$ inside $\mathcal{W}'$, and outside of $\mathcal{W}''$ we obtain a $1$-homogeneous Hamiltonian isotopy $\Phi^t_{H}$ on $\R^{2n} \setminus \{0\}$, generated by $H_t(x,r) = h_t(x) r$.  It satisfies $\Phi^1_{H}(\mathcal{W}'') = \mathcal{V}''$ where $\mathcal{W}'' \subset \mathcal{W}' \subset \partial B^{2n}(\epsilon)$ and $\mathcal{V}'' \subset \mathcal{V}' \subset \partial U$. Now, replace $U$ with $(\Phi^1_{H})^{-1}(U)$. In particular, a part of our boundary is a part of a standard contact sphere of radius $\sqrt{\epsilon/ \pi}$.

\textbf{Part 2:}  Take a neighborhood $B^{2n}(p, \epsilon')$ of point $p \in \partial U \cap \partial B^{2n}(\epsilon)$ so that $B^{2n}(p, 2\epsilon')$ is contained in the symplectization $\mathcal{W}'' \times (0, +\infty) \subset \R^{2n}$, where $\mathcal{W}'' \subset \partial U \cap \partial B^{2n}(\epsilon)$. Let $\rho:B^{2n}(\epsilon) \to [0,1]$ be a cutt-off function so that $\rho \equiv 1$ on $B^{2n}(p, \epsilon') \cap B^{2n}(\epsilon)$ and $\rho \equiv 0$ outside of $B^{2n}(p, 2\epsilon')$ in $B^{2n}(\epsilon)$. From the proof of Theorem \ref{thm:main} we see that $\Psi_{k}$ is realized by Hamiltonian isotopies generated by $H_t^k: \R^{2n} \to \R$, after apealing again to the connectednes of the symplectic embeddings of a disc into connected compact surface, and by Lemma \ref{lma:maksim} we have $H_t^{k+1} \vert_{ \Phi^t_{H_k}(\delta_k' B^{2n}(\epsilon) )} = H_t^{k}\vert_{ \Phi^t_{H_k} (\delta_k' B^{2n}(\epsilon))}$. Now, modify $H_t^{k}$ by setting

$$
\widetilde{H}_t^{k}(x) = \rho((\Phi^t_{H_k})^{-1}(x)) H_t^{k}(x).
$$

Note that $\widetilde{H}_t^{k} \vert_{\delta_k' B^{2n}(\epsilon)}$ is zero if $\delta_k' B^{2n}(\epsilon) \cap B^{2n}(p, 2\epsilon') = \emptyset$. Now, define the symplectic embedding:
$$
\begin{aligned}
\Phi: \mathrm{int}(B^{2n}(\epsilon)) &\to \R^{2n} \\
\delta_k'B^{2n}(\epsilon) \ni x &\mapsto \Phi^1_{\widetilde{H}_k}(x).
\end{aligned}
$$ 

 On $B^{2n}(p, \epsilon') \cap \mathrm{int} B^{2n}(\epsilon)$, $\Phi$ coincides with $\Psi_k = \Phi^1_{H_k}$, and it is identity on $\mathrm{int} (B^{2n}(\epsilon)) \setminus B^{2n}(p, 2\epsilon')$. 

Since the part of the boundary of $\Phi^1_{H_t}(B^{2n}(\epsilon))$ coincides with the part of the boundary of $W_1 \times_2 \cdots \times_2 W_n = \Psi(B^{2n}(\epsilon))$ we finish the proof by taking domains $W_i$ so that each point of the boundary $\partial(W_1 \times_2 \cdots \times_2 W_n)$ has neighborhood of Hausdorff dimension $2n-1 + \alpha$.

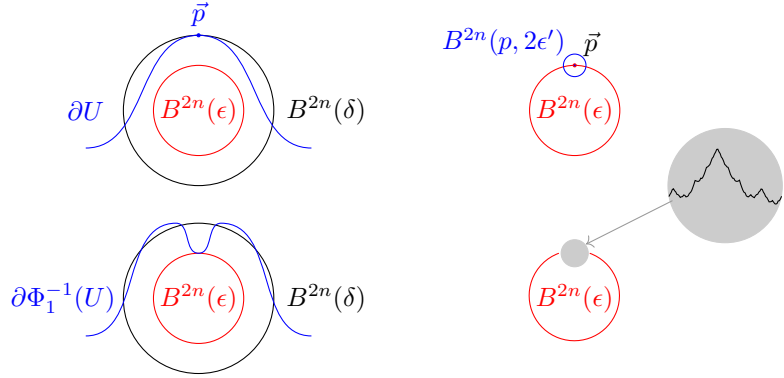
\begin{figure}[h!]
\centering
\begin{tikzpicture}
\draw (-2,2) circle (1cm);
\draw[red] (-2,2) circle (0.6 cm);
\draw[blue] (-3.5,1.5) to [out =0, in = 180] (-2, 3) to [out = 0, in = 180] (-0.5,1.5);
\node[red] at (-2,2) {$B^{2n}(\epsilon)$};
\node[blue] at (-3.5,2) {$\partial U$};
\node at (-0.3, 2) {$B^{2n}(\delta)$};

\draw (-2,-0.5) circle (1cm);
\draw[red] (-2,-0.5) circle (0.6 cm);
\draw[blue] (-3.5,-1) to [out =0, in = 180] (-2.3, 0.5) to [out =0, in = 180] (-2, 0.1)  to [out = 0, in = 180] (-1.7, 0.5)to [out = 0, in = 180] (-0.5,-1);
\node[red] at (-2,-0.5) {$B^{2n}(\epsilon)$};
\node[blue] at (-3.8,-0.5)  {$\partial \Phi_1^{-1}(U)$};
\node at (-0.3, -0.5) {$B^{2n}(\delta)$};
\draw[blue, fill=blue] (-2,3) circle (0.2mm);
\node[blue] at (-2,3) [above] {$\vec{p}$};

\draw[red] (3,2) circle (0.6 cm);
\node[red] at (3,2) {$B^{2n}(\epsilon)$};
\draw[red, fill=blue] (3,2.6) circle (0.2mm);
\node at (3,2.6) [above right] {$\vec{p}$};
\draw[blue] (3,2.6) circle (0.15cm);
\draw[blue] node at (3,2.6) [above left] {$B^{2n}(p, 2\epsilon')$};

\draw[red] (3.2,0.1) arc (70:-250:0.6);
\node[red] at (3,-0.5) {$B^{2n}(\epsilon)$};
\draw[black!20!white, fill=black!20!white] (5, 1) circle (7.6mm);
\draw[->, black!40!white] (4.3, 0.8) -- (3.15, 0.22);
\draw[black!20!white, fill=black!20!white] (3, 0.1) circle (1.8mm);
\begin{scope}[scale = 1.5, shift = {(4.835, 1)}]
\draw[domain=-2:-1, samples=100, /pgf/fpu, /pgf/fpu/output format=fixed] 
        plot (\x, {\weierstrass});
\end{scope}

\end{tikzpicture}
\caption{On left: Hamiltonian deformation of the neigborhood of $\vec{p}$ to the boundary of $B^{2n}(\epsilon)$. On the right: cutting off $\Psi:\mathrm{int}(B^{2n}(\epsilon)) \to \mathrm{int}(W_1\times_2 \cdots \times_2 W_n)$ in a neighborhood of a point.}
\end{figure}

\subsection{Proof of Proposition \ref{prop:dyn}}

For a locally Lipschitz function $f: \R^{m} \to \R$ generalized gradient is well defined at every point $x$, which we denote by $\partial f (x)$ (see \cite{Cl83}). A generalized gradient is a non-empty convex compact subset of $\R^m$, moreover, $\partial f(x) = \{ \nabla f(x)\}$ when $f$ is differentiable at $x$. Let $\Omega \subset \C$ be a star-shaped set with Lipschitz boundary. We consider the characteristic equation:
$$
\begin{aligned}
\gamma'(t) &\in i \partial H_{\Omega} (\gamma(t)) \hspace{2mm} \text{a.e.}, \\
\gamma(0) &= z_0 \in \partial \Omega,
\end{aligned}
$$
where $H_{\Omega}$ is $2$-homogeneous function such that $\Omega = \{H_{\Omega} \leq 1\}$. Since $\partial \Omega$ is $1$-dimensional, the previous problem has a unique solution, which is extendable to the whole $\R$ and periodic, with minimal period equal to the area of $\Omega$. This can be extended to the $1$-homogeneous flow $\Phi^t_{\Omega}$ on $\R^2\setminus \{0\}$ which extends continuously to the origin by $\Phi^t_{\Omega}(0)=0$. Since $H_{W_1 \times_2 \cdots \times_2 W_n} = H_{W_1} + \cdots + H_{W_n}$, we have that the characteristic on $\partial(W_1 \times_2 \cdots \times_2 W_n)$ splits. This further implies:
$$
\Phi_{H_{W_1 \times_2 \cdots \times_2 W_n}}^t(z_1,\dots,z_n) = (\Phi^t_{H_{W_1}}(z_1), \dots, \Phi^t_{H_{W_n}}(z_n)).
$$
Now we construct the homeomorphism $\Psi: \partial E(a_1, \dots, a_n) \to \partial(W_1 \times_2 \cdots \times_2 W_n)$:
$$
\Psi(r_1 e^{i 2\pi \theta_1}, \dots r_n e^{i 2\pi \theta_n}) = \left(\Phi_{H_{W_1}}^{\theta_1 a_1}(r_1), \dots, \Phi_{H_{W_n}}^{\theta_n a_n}(r_n)\right).
$$
One easily checks that:
$$
\Psi (e^{i 2\pi t /a_1} z_1, \dots e^{i 2\pi t /a_n} z_n) = \Phi^t_{H_{W_1 \times_2 \cdots \times_2 W_n}} (\Psi(z_1, \dots, z_n)).
$$

\subsection{Proof of Proposition \ref{prop:boundary_minimal}}

Assume that there exists an open set $U$ with $U \cap \partial (W_1 \times_2 \cdots \times_2 W_n) \neq \emptyset$ such that 
$$
c_1^{EH}(W_1 \times_2 \cdots \times_2 W_n \setminus U) = c_1^{EH}(W_1 \times_2 \cdots \times_2 W_n).
$$
From Theorem \ref{thm:main} we have that $c_1^{EH}(W_1 \times_2 \cdots \times_2 W_n) = c_1^{EH}(B^{2n}(a)) = a$, where $a$ is the area of $W_i$. Let $(z_1,\dots, z_n) \in U \cap \partial (W_1 \times_2 \cdots \times_2 W_n) $. Since $U$ is open, we can assume that $z_i \neq 0$ for all $i$. Every domain has a defining continuous function $H_{W_i} (z)= f_i(z/|z|) |z|^2$ for some $f_i: S^1 \to (0, + \infty)$. Alter each $f_i$ to $g_i$ in the neighborhood of $z_i/|z_i|$ so that:
\begin{itemize}
\item $g_i \geq f_i$,
\item Areas of $W_i' = \{ H'_i \leq 1\} $ are equal to $a' < a$, where $H'_i(z) = g_i(z/|z|) |z|^2$,
\item $W_1 \times_2 \cdots \times_2 W_n \setminus U \subset W_1' \times_2 \cdots \times_2 W_n'$ (see Figure \ref{fig:boundary_minimal}). 
\end{itemize}
 
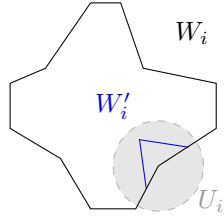
\begin{figure}[h!]
\centering
\begin{tikzpicture}

\draw[black!30!white, fill=black!10!white, dashed] (0.6, -0.8) circle (6mm);
\draw[blue] (0.44, -1.1) -- (0.35, -0.45) -- (1, -0.55);
\draw[black!40!white] node at (1,-1) [below right] {$U_i$};
\draw (-0.3, 1.37) -- (0.1, 1.37) -- (0.4, 0.5)-- (1.37, 0.3) -- (1.37, -0.3) -- (0.6, -0.8) -- (0.3, -1.37) -- (-0.3, -1.37) -- (-0.7, -0.7) -- (-1.37, -0.3) -- (-1.37, 0.3) -- (-0.9, 0.5) -- (-0.3, 1.37);

\draw[blue] node at (0,0) {$W_i'$};
\draw[dashed] node at (0.7, 0.7) [above right] {$W_i$};

\end{tikzpicture}
\caption{Altered star-shaped domains $W_i' \subset W_i$, and a scaled image $U_i$ of the intersection of $U$ with $i$-th complex line.}
\label{fig:boundary_minimal}
\end{figure}

By applying monotonicity of $c_1^{EH}$, and Theorem \ref{thm:main} to $W_1' \times_2 \cdots \times_2 W_n'$ we have:
$$
a = c_1^{EH}(W_1 \times_2 \cdots \times_2 W_n \setminus U) \leq c_1^{EH}(W_1' \times_2 \cdots \times_2 W_n') = c_1^{EH}(B^{2n}(a')) = a',
$$
which is a contradiction since $a' < a$. Here, we have used that Ekeland-Hofer capacities are preserved under the symplectomorphisms of the interiors of star-shaped domains. This follows from \cite{GR24}, where it is shown that Ekeland-Hofer capacities coincide with Gutt-Hutchings capacities, which satisfy the invariance under the symplectomorphism of the interior.

{\footnotesize
\bibliography{citations}
\bibliographystyle{alpha}
}
\Address
\end{document}